\documentclass[12pt]{emsprocart}
\usepackage[all]{xy}
\SelectTips{cm}{}
\usepackage{dsfont}
\numberwithin{equation}{section}

\theoremstyle{plain}
\newtheorem{lemma}{Lemma}[section]
\newtheorem{prop}[lemma]{Proposition}
\newtheorem{theorem}[lemma]{Theorem}

\theoremstyle{definition}
\newtheorem{defn}[lemma]{Definition}

\newtheorem{eg}[lemma]{Example}

\theoremstyle{remark} 
\newtheorem{remark}[lemma]{Remark}

\newcommand{\bL}{\mathbf L}
\newcommand{\bZ}{\mathbf Z}

\newcommand{\mcV}{\mathcal V}
\newcommand{\mcW}{\mathcal W}
\newcommand{\mcX}{\mathcal X}
\newcommand{\sfC}{\mathsf{C}}
\newcommand{\sfD}{\mathsf{D}}
\newcommand{\sfK}{\mathsf{K}}
\newcommand{\sfS}{\mathsf S}
\newcommand{\sfT}{\mathsf T}
\newcommand{\fc}{\mathfrak{c}}
\newcommand{\fm}{\mathfrak{m}}
\newcommand{\fp}{\mathfrak{p}}
\newcommand{\fq}{\mathfrak{q}}
\newcommand{\cl}{\operatorname{cl}}
\newcommand{\End}{\operatorname{End}}
\newcommand{\Ext}{\operatorname{Ext}}
\renewcommand{\Gamma}{\varGamma}
\renewcommand{\Lambda}{\varLambda}

\newcommand{\fHom}{\operatorname{\mathcal{H}\!\!\;\mathit{om}}}
\newcommand{\Hom}{\operatorname{Hom}}
\newcommand{\Id}{\operatorname{Id}}
\newcommand{\KacInj}{\operatorname{\mathsf{K}_{\mathrm{ac}}\mathsf{Inj}}}
\newcommand{\Ker}{\operatorname{Ker}}
\newcommand{\KInj}{\operatorname{\mathsf{KInj}}}
\newcommand{\Coloc}{\operatorname{Coloc}}
\newcommand{\Loc}{\operatorname{Loc}}
\newcommand{\sHom}{\underline\Hom}
\newcommand{\PHom}{\operatorname{PHom}}
\newcommand{\mmod}{\operatorname{\mathsf{mod}}}
\newcommand{\Mod}{\operatorname{\mathsf{Mod}}}
\newcommand{\one}{\mathds 1}
\newcommand{\proj}{\operatorname{\mathsf{proj}}}
\newcommand{\res}{\operatorname{res}}
\newcommand{\Spec}{\operatorname{Spec}}
\newcommand{\stmod}{\operatorname{\mathsf{stmod}}}
\newcommand{\StMod}{\operatorname{\mathsf{StMod}}}
\newcommand{\cosupp}{\operatorname{cosupp}}
\newcommand{\supp}{\operatorname{supp}}

\title{Module categories for finite group algebras}
\author{David J. Benson, Srikanth B. Iyengar and Henning Krause}
\contact[bensondj@maths.abdn.ac.uk]
{David J. Benson, 
Institute of Mathematics, 
Fraser Noble Building, 
University of Aberdeen,
Aberdeen AB24 3UE, UK}
\contact[iyengar@math.unl.edu]
{Srikanth B. Iyengar, 
Department of Mathematics, 
University of Nebraska, 
Lincoln NE 68588-0130, USA}
\contact[hkrause@math.uni-bielefeld.de]{Henning Krause, 
Universit\"at Bielefeld, 
Fakult\"at f\"ur Mathematik, 
Postfach 100 131, 
D-33 501 Bielefeld, 
Germany}

\begin{document}

\begin{abstract}
This survey article is intended as an introduction to the recent categorical classification theorems of the three authors, restricting to the special case of the category of modules for a finite group.
\end{abstract}



\section*{Introduction}

The purpose of this paper is to explain the recent work of the three
authors on the classification of localising and colocalising
subcategories of triangulated categories
\cite{Benson/Iyengar/Krause:2008a, Benson/Iyengar/Krause:bik2,
  Benson/Iyengar/Krause:bik3, Benson/Iyengar/Krause:bik4}, and to put
it into context. In order to be as concrete as possible, we restrict
our attention to categories associated with the modular representation
theory of finite groups. A more leisurely discussion, filling in
requisite background from commutative algebra and triangulated
categories, is given in \cite{Benson/Iyengar/Krause:bik6}.

Let $G$ be a finite group and let $k$ be an algebraically closed 
field of characteristic $p>0$.
With a few exceptions in characteristic two, if the Sylow $p$-subgroups
of $G$ are non-cyclic then the group algebra $kG$ has wild representation
type. We therefore do not hope to classify the indecomposable
finitely generated $kG$-modules, and so we are left with several options.
We can restrict the kinds of modules that we're interested in;
we can look for properties of modules that we can prove without
the need for a classification; or we look for coarser classification
theorems. Examples abound of all three types of approach; we shall
be examining the last option.

The first indication that it might be possible to make some sort of
categorical classification theorem came with Mike Hopkins' 1987
classification \cite{Hopkins:1987a} of the thick subcategories of the
derived category $\sfD^b(\proj(R))$ of perfect complexes over a
commutative Noetherian ring $R$. His interest was in the nilpotence
theorem for the stable homotopy category, for which he regarded the
derived category of perfect complexes as a toy model. The
parametrisation of the thick subcategories was by specialisation
closed sets of prime ideals in $R$.

Neeman's 1992 paper \cite{Neeman:1992a} took up Hopkins' work, 
clarified it and went on to classify the localising subcategories of the 
unbounded derived category of $R$-modules $\sfD(\Mod(R))$.
This is the appropriate big 
category whose compact objects are the perfect complexes
$\sfD^b(\proj(R))$.
The parametrisation for the localising subcategories was by 
\emph{all} subsets
of the set of prime ideals of $R$, not just the specialisation
closed ones.

In 1997, the first author together with Carlson and Rickard
\cite{Benson/Carlson/Rickard:1997a} proved the analogue of Hopkins'
theorem for the stable category $\stmod(kG)$ of finitely generated
$kG$-modules, over an algebraically closed field $k$ of characteristic
$p$.  The parametrisation for the thick subcategories was by
specialisation closed subsets of homogeneous primes in the cohomology
ring $H^*(G,k)$, ignoring the maximal ideal of all positive degree
elements.

The corresponding big category is the stable category 
$\StMod(kG)$ of all $kG$-modules; the full subcategory of 
compact objects in this 
is the finitely generated stable module category $\stmod(kG)$.
It was expected by analogy with $\sfD(\Mod(R))$ that the 
classification of localising subcategories of $\StMod(kG)$ would
be by all subsets of the set of non-maximal homogeneous primes in
$H^*(G,k)$. But this turned out to be very hard to prove, and it
was not until a couple of years ago that we
managed to achieve this classification \cite{Benson/Iyengar/Krause:bik3}.

More recently, in 2009 Neeman \cite{Neeman:coloc} classified the
colocalising subcategories of $\sfD(\Mod(R))$; the corresponding
classification for $\StMod(kG)$ is given in
\cite{Benson/Iyengar/Krause:bik4}.

Other classifications follow similar models.  The work of Hovey,
Palmieri and Strickland \cite{Hovey/Palmieri/Strickland:1997a} and the
series of papers
\cite{Benson/Iyengar/Krause:2008a,Benson/Iyengar/Krause:bik2,
  Benson/Iyengar/Krause:bik3,Benson/Iyengar/Krause:bik4} lay the
general foundations. For background on the theory of support varieties
we refer to Solberg's survey \cite{Solberg:2006a}.
\bigskip

\begin{center}
\sc Contents.
\end{center}


\ \,1. The stable module category

\ \,2. Thick subcategories of $\stmod(kG)$

\ \,3. The derived category

\ \,4. Rickard idempotent modules and functors

\ \,5. Classification of tensor ideal thick subcategories

\ \,6. Localising subcategories of $\StMod(kG)$

\ \,7. The category $\KInj(kG)$

\ \,8. Support for triangulated categories

\ \,9. Tensor triangulated categories

10. The local-global principle

11. Stratifying triangulated categories

12. Graded polynomial algebras

13. The BGG correspondence

14. The Koszul construction

15. Quillen stratification

16. Applications

17. Costratification

\section{The stable module category}

We write $\Mod(kG)$ for the module category of $kG$. The objects are
the left $kG$-modules and the morphisms are the module
homomorphisms. This is an abelian category. We write $\mmod(kG)$ for
the full subcategory of finitely generated $kG$-modules. We refer the
reader to \cite{Benson:1991a} for basic constructions and facts
concerning $kG$-modules. The following statement is well known for
finitely generated modules, but is true more generally.

\begin{lemma}\label{le:pr=inj}
Projective and injective $kG$-modules coincide.
\end{lemma}
\begin{proof}
It is an easy exercise using the group basis to show that $kG$ is
self-dual as a $kG$-module. This implies that every projective 
module is injective. For the converse, we note that every
$kG$-module $M$ embeds in a free $kG$-module, via the
map $M \to kG\otimes_k M$ sending $m$ to $\sum_{g\in G}g\otimes g^{-1}m$. 
Here, we make $kG\otimes_k M$ into a
$kG$-module via $g(g'\otimes m)=gg'\otimes m$; thus a $k$-basis
of $M$ gives a free $kG$-basis for $kG\otimes_k M$. If $M$ is
injective then this embedding splits and $M$ is a summand of a free module,
hence projective.
\end{proof}

Projective $kG$-modules are well understood. Every projective is a
direct sum of finitely generated projective indecomposables, and every
projective indecomposable is the projective cover of a simple module.

To work ``modulo projectives'', we use the \emph{stable module category}
$\StMod(kG)$.
This has the same objects as the module category, but the morphisms
are given by
\[ \sHom_{kG}(M,N)=\Hom_{kG}(M,N)/\PHom_{kG}(M,N) \]
where $\PHom_{kG}(M,N)$ is the linear subspace of maps that factor
through some projective module.

We write $\stmod(kG)$ for the full subcategory of $\StMod(kG)$ 
whose objects are the modules stably isomorphic to finitely generated modules. 
Notice that if
a map between finitely generated modules factors through some projective
module then it factors through a finitely generated projective module.

The subcategory $\stmod(kG)$ is 
distinguished categorically as the \emph{compact objects}
in $\StMod(kG)$.

\begin{defn}
An object $M$ in $\StMod(kG)$ 
is said to be \emph{compact} if given any small
coproduct $\bigoplus_\alpha M_\alpha$, the natural map
\[ \bigoplus_\alpha\sHom_{kG}(M,M_\alpha)\to
\sHom_{kG}(M,\bigoplus_\alpha M_\alpha) \]
is an isomorphism.
\end{defn}

Although $\StMod(kG)$ and $\stmod(kG)$ are not abelian categories,
they come with a natural structure of triangulated category.
The ``shift'' is given by $\Omega^{-1}$, the functor assigning
to each module $M$ the cokernel of an embedding of $M$ into an injective
module. At the level of $\Mod(kG)$ this is not functorial, but for
$\StMod(kG)$ it is a functorial self-equivalence whose inverse is the
functor $\Omega$ taking a module $M$ to the kernel of a surjection from
a projective module onto $M$.

The distinguished triangles
\[ A \to B \to C \to \Omega^{-1}(A) \]
are by definition those isomorphic to diagrams coming from a
short exact sequence 
\[ 0\to A\to B \to C\to 0 \] 
of modules as follows. Given such a short exact sequence, we embed $B$
into an injective module $I$ and obtain the following commutative
diagram with exact rows:
\[ \xymatrix{0\ar[r]&A\ar[r]\ar@{=}[d]&B\ar[r]\ar[d]&C\ar[r]\ar[d]&0\\
0\ar[r]&A\ar[r]&I\ar[r]&\Omega^{-1}(A)\ar[r]&0.} \] This gives us a
map $C\to \Omega^{-1}A$ which is well defined in $\StMod(kG)$.

The right hand square in this diagram gives us a short exact sequence
\[ 0 \to B \to C\oplus I \to \Omega^{-1}(A) \to 0. \]
Since $C$ is isomorphic to $C\oplus I$ in $\StMod(kG)$,
applying the construction again gives rise to the rotated triangle
\[ B\to C\to \Omega^{-1}(A) \to \Omega^{-1}(B). \]

Most of the axioms for a triangulated category are easy to verify.
The only one that needs any comment is the octahedral axiom,
which is the expression in $\StMod(kG)$ of the
third isomorphism theorem in $\Mod(kG)$. More details can be
found in Buchweitz \cite{Buchweitz:1986}, and in Happel \cite{Happel:1988a}.

\section{Thick subcategories of $\stmod(kG)$}

\begin{defn}
For any triangulated category $\sfT$, a \emph{thick subcategory} $\sfS$
is a full triangulated subcategory that is closed under taking finite direct
sums and summands.
\end{defn}

It is worth spelling out what this means for the stable module category.
Being a full triangulated subcategory means that the subcategory is
closed under $\Omega$ and $\Omega^{-1}$, and it has the 
\emph{two in three property}: given a triangle in the stable module
category for which two of the objects are in $\sfS$, so is the third.

Pulling back to $\mmod(kG)$,
thick subcategories of $\stmod(kG)$ are in one to one correspondence
with the subcategories $\fc$ 
of $\mmod(kG)$ with the following properties:\begin{enumerate}
\item All projective modules are in $\fc$.
\item $\fc$ is closed under finite sums and summands.
\item If $M$ is in $\fc$ then so are $\Omega(M)$ and $\Omega^{-1}(M)$.
\item If $ 0\to A\to B\to C\to 0$ is a short exact sequence of modules
  with two of $A$, $B$ and $C$ in $\fc$ then so is the third.
\end{enumerate}

The clue to finding thick subcategories of $\stmod(kG)$ comes from the
theory of varieties for modules, which we now briefly describe. A
fuller treatment can be found in \cite{Benson:1991b}. The cohomology
ring $H^*(G,k)$ is defined to be $\Ext^*_{kG}(k,k)$ where $k$ denotes
the trivial representation. This is a finitely generated graded
commutative ring, and we write $V_G$ for its maximal ideal spectrum.

If $M$ is a finitely generated $kG$-module, we have a ring homomorphism
\[ H^*(G,k)=\Ext^*_{kG}(k,k) \xrightarrow{M \otimes_k -} \Ext^*_{kG}(M,M) \]
given by tensoring Yoneda extensions with $M$. As usual, we are tensoring
over $k$ with diagonal group action, which is exact. We write $I_M$ for
the kernel of this homomorphism, and $V_G(M)$ for the subvariety of
$V_G$ determined by this ideal. Since $I_M$ is a homogeneous ideal,
$V_G(M)$ is a closed homogeneous subvariety of the homogeneous variety $V_G$.
The following theorem summarises some
of the main properties of these varieties; not all of them are easy
to prove.

\begin{theorem}
\label{th:VGM}
Let $M$, $M_1$, $M_2$, $M_3$ be finitely generated $kG$-modules.
\begin{enumerate}
\item $V_G(M)=\{0\}$ if and only if $M$ is a projective module.
\item More generally, the dimension of the variety $V_G(M)$
determines the polynomial rate of growth of 
a minimal resolution of $M$.
\item $V_G(\Hom_{k}(M,k))=V_G(M)$.
\item $V_G(M_1\oplus M_2)=V_G(M_1)\cup V_G(M_2)$.
\item $V_G(M_1\otimes_k M_2)=V_G(M_1)\cap V_G(M_2)$.
\item If $0\to M_1\to M_2 \to M_3 \to 0$ is a short exact sequence
then for $i=1,2,3$, $V_G(M_i)$ is contained in the union of the
varieties of the other two modules.
\item If $0\ne\zeta\in H^n(G,k)$ is represented by a cocycle
$\hat\zeta\colon\Omega^n(k)\to k$, we write $L_\zeta$ for the kernel
of $\hat\zeta$. Then we have
$V_G(L_\zeta)=V_G\langle\zeta\rangle$, the variety determined by
the principal ideal $\langle\zeta\rangle$.
\item Given a closed homogeneous subvariety $V\subseteq V_G$, there
  exists a finitely generated module $M$ with $V_G(M)=V$.  Namely, if
  $V=V_G\langle \zeta_1,\dots,\zeta_n\rangle$ then we may take
  $M=L_{\zeta_1}\otimes_k\dots\otimes_k L_{\zeta_n}$ and use
  properties {\rm (v)} and {\rm (vii)}.
\end{enumerate}
\end{theorem}

\begin{defn}
We write $\mcV_G$ for the collection of closed homogeneous irreducible
subvarieties of $V_G$ (including $\{0\})$, 
or equivalently the spectrum of homogeneous prime
ideals $\fp\subseteq H^*(G,k)$ (including the maximal ideal $\fm$ 
of positive degree elements). We say that a subset $\mcV\subseteq\mcV_G$ is
\emph{specialisation closed} if $\fp\in\mcV$, $\fq\in\mcV_G$
and $\fq\supseteq \fp$ imply $\fq\in\mcV$.

If $\mcV$ is a specialisation closed subset of $\mcV_G$ then we write
$\fc_\mcV$ for the full subcategory of $\stmod(kG)$ 
consisting of modules $M$ such that
$V_G(M)$ is a finite union of elements of $\mcV$ (i.e., each irreducible
component of $V_G(M)$ is in $\mcV$).
\end{defn}

\begin{lemma}\label{le:CV}
If $\mcV$ is a specialisation closed subset of $\mcV_G$ then 
$\fc_\mcV$ is a thick subcategory of $\stmod(kG)$.
\end{lemma}
\begin{proof}
This follows directly from Theorem~\ref{th:VGM}.
\end{proof}

We now come to an issue which will reappear in later contexts. Namely,
the thick subcategories of $\stmod(kG)$ appearing in Lemma~\ref{le:CV}
are closed under tensor products with finitely generated $kG$-modules, by
part (iii) of Theorem~\ref{th:VGM}.

\begin{lemma}
If $\fc$ is a thick subcategory of $\stmod(kG)$ then the following 
are equivalent:
\begin{enumerate}
\item $\fc$ is closed under tensor product with 
finitely generated $kG$-modules.\smallskip
\item $\fc$ is closed under tensor product with simple 
$kG$-modules.
\end{enumerate}
These conditions are automatically satisfied if $G$ is a finite
$p$-group.
\end{lemma}
\begin{proof}
The equivalence of (i) and (ii) follows from the fact that
every finitely generated 
module has a finite filtration in which the filtered quotients
are simple modules.
If $G$ is a finite $p$-group then the trivial module is the only 
simple $kG$-module.
\end{proof}

\begin{defn}
We say that a thick subcategory of $\stmod(kG)$ is \emph{tensor ideal}, or \emph{tensor closed},
if the equivalent conditions of the lemma are satisfied.
\end{defn}

We can now state the classification theorem. This was proved
in Theorem~3.4 of Benson, Carlson and Rickard 
\cite{Benson/Carlson/Rickard:1997a} in case $k$ is an algebraically
closed field, and in Theorem~11.4 of 
Benson, Iyengar and Krause \cite{Benson/Iyengar/Krause:bik3}
for general fields $k$.

\begin{theorem}\label{th:thick}
There is a one to one correspondence between tensor ideal thick 
subcategories of $\stmod(kG)$ and non-empty specialisation closed subsets
of $\mcV_G$. Under this correspondence, $\mcV$ corresponds to $\fc_\mcV$.
\end{theorem}

Note that the condition that $\mcV$ is non-empty is equivalent to the
condition that $\mcV$ contains $\{0\}$. This subset plays no role for
$\stmod(kG)$ but will make an appearance later.

If we remove the tensor ideal condition, then the classification 
becomes much harder. 
One example of a thick subcategory which is
usually not tensor ideal is the full subcategory of modules in
the principle block.
The \emph{principal block} $B_0(kG)$ is the block containing the
trivial module. So one might ask whether the thick subcategories of
$\stmod(B_0(kG))$  are all of the form 
$\fc_\mcV \cap \stmod(B_0(kG))$.
The obstruction to this being true
is given in terms of the \emph{nucleus}, 
defined as follows.

\begin{defn}
The \emph{nucleus} $Y_G$ of a finite group $G$ is the union of the images of
$\res^*_{G,H}\colon V_H\to V_G$, as $H$ runs over the subgroups of
$G$ such that $C_G(H)$ is not $p$-nilpotent. Recall that a finite
group is said to be 
\emph{$p$-nilpotent} if it has a normal subgroup whose order is prime to $p$
and whose index is a power of $p$.
\end{defn}

\begin{theorem}[Benson, Carlson and 
Robinson \cite{Benson/Carlson/Robinson:1990a}; 
Benson \cite{Benson:1995a}]\label{th:nucleus}
The subvariety $Y_G$ of $V_G$ is equal to the union of the $V_G(M)$ as
$M$ runs over the finitely generated modules $M$ in $B_0(kG)$
satisfying $H^*(G,M)= 0$. 

In particular, every non-projective
module in the principal block has non-trivial cohomology if and only if the
centraliser of every element of order $p$ in $G$ is $p$-nilpotent.
\end{theorem}

\begin{theorem}[Benson, Carlson and 
Rickard \cite{Benson/Carlson/Rickard:1997a}]\label{th:no-nucleus}
If $Y_G$ is empty or equal to $\{0\}$ then there is a one to one correspondence
between the thick subcategories of $\stmod(B_0(kG))$ and the non-empty
subsets of $V_G$. Under this correspondence, $\mcV$ corresponds to
$\fc_\mcV\cap\stmod(B_0(kG))$.
\end{theorem}

If $Y_G$ is bigger than $\{0\}$ then it appears to be easy to manufacture
infinite collections of thick subcategories supported on each line
through the origin in $Y_G$. It appears hopeless to classify these
thick subcategories. However, if we work modulo the modules supported
inside the nucleus, we again obtain a classification of thick subcategories.

\section{The derived category}

The anomalous role of the origin in the classification of thick 
subcategories of $\stmod(kG)$ can be repaired by moving to a 
slightly bigger category, namely the \emph{derived category}
$\sfD^b(\mmod(kG))$. Recall that the objects in this category are
the finite chain complexes of finitely generated $kG$-modules, and the arrows
are homotopy classes of maps of complexes, with the quasi-isomorphisms
inverted. A \emph{quasi-isomorphism} is a map of complexes that induces
an isomorphism in homology. The category $\sfD^b(\mmod(kG))$ is a 
triangulated category, in which the triangles are formed using mapping
cones.

We write $\sfK^b(\proj(kG))$ for the thick subcategory of
$\sfD^b(\mmod(kG))$ whose objects are the \emph{perfect complexes},
i.e., the complexes quasi-isomorphic to a bounded complex of finitely
generated projective $kG$-modules. Buchweitz~\cite{Buchweitz:1986}
(see also Rickard \cite{Rickard:1989a}) has defined a functor from
$\sfD^b(\mmod(kG))$ to $\stmod(kG)$ which is essentially surjective
and has kernel $\sfK^b(\proj(kG))$:
\[ \sfK^b(\proj(kG)) \to \sfD^b(\mmod(kG)) \to \stmod(kG). \]

The thick subcategory $\sfK^b(\proj(kG))$ is the unique minimal tensor
ideal thick subcategory of $\sfD^b(\mmod(kG))$. This allows us to
extend the theory of support varieties to $\sfD^b(\mmod(kG))$ in such
a way that an object has the origin in its support if and only if
it is non-zero. As before, if $\mcV$ is a specialisation closed
subset of $\mcV_G$, we write $\fc_\mcV$ for the thick subcategory
of $\sfD^b(\mmod(kG))$ consisting of objects whose support is a finite
union of elements of $\mcV$.
So the following theorem is an easy consequence of 
Theorem~\ref{th:thick}.

\begin{theorem}
There is a one to one correspondence between tensor ideal thick 
subcategories of $\sfD^b(\mmod(kG))$ and specialisation closed subsets
of $\mcV_G$. Under this correspondence, $\mcV$ corresponds to $\fc_\mcV$.
\end{theorem}

Next, we observe that $G$ is $p$-nilpotent if and only if $k$ is the
only simple module in the principal block. This allows us to extend
Theorem~\ref{th:no-nucleus} as follows.

\begin{theorem}
If $Y_G$ is empty (i.e., if $G$ is $p$-nilpotent) 
then there is a one to one correspondence
between the thick subcategories of $\sfD^b(\mmod(B_0(kG)))$ and 
subsets of $V_G$. Under this correspondence, $\mcV$ corresponds to
$\fc_\mcV\cap\sfD^b(\mmod(B_0(kG)))$.
\end{theorem}

Again, if $Y_G$ is non-empty, we can work modulo objects supported 
inside the nucleus and obtain a classification of thick subcategories,
but this returns us to quotients of $\stmod(B_0(kG))$ so nothing new
is gained.

\section{Rickard idempotent modules and functors}\label{se:Rickard}

The proof of Theorem~\ref{th:thick} depends
in an essential way on the construction of certain 
infinitely generated modules, and the development of
a theory of support varieties in this context.

The modules in question are \emph{Rickard
idempotent modules}. Corresponding to any specialisation closed
subset $\mcV$ of $\mcV_G$, Rickard constructs two idempotent functors
on $\StMod(kG)$. He writes these as $E_\mcV$ and $F_\mcV$, but for
consistency with the notation of our series of papers we shall
use the notation $\Gamma_\mcV$ and $L_\mcV$. More generally, given
any thick subcategory $\fc$ of $\stmod(kG)$, we have functors
$\Gamma_\fc$ and $L_\fc$, and functorial triangles
\begin{equation}\label{eq:Rickard}
\Gamma_\fc(M) \to M \to L_\fc(M). 
\end{equation}

The defining properties of these triangles are given in terms of
\emph{localising subcategories} of $\StMod(kG)$.

\begin{defn}
Let $\sfT$ be a triangulated category with small products and coproducts.
A \emph{localising subcategory} of $\sfT$ is a thick subcategory
that is closed under coproducts, while a \emph{colocalising subcategory}
of $\sfT$ is a thick subcategory that is closed under products.
\end{defn}
The triangle of functors \eqref{eq:Rickard} is characterised by the 
following properties:
\begin{enumerate}
\item $\Gamma_\fc(M)$ is in the localising subcategory $\Loc(\fc)$ generated
by $\fc$.
\item If $N$ is in $\fc$ then $\sHom_{kG}(N,L_\fc(M))=0$.
\end{enumerate}
The functors $\Gamma_\fc$ and $L_\fc$ preserve small coproducts,
see for example \cite[Corollary~6.5]{Benson/Iyengar/Krause:2008a}.\smallskip

If $\fc$ is a tensor ideal thick subcategory then we have
\[ \Gamma_\fc(k) \otimes_k M \cong \Gamma_\fc(M), \qquad
L_\fc(k) \otimes_k M \cong L_\fc(M). \]
If $\fc=\fc_\mcV$ then we abbreviate $\Gamma_{\fc_\mcV}$ to $\Gamma_\mcV$
and $L_{\fc_\mcV}$ to $L_\mcV$, so that the Rickard triangle takes the form
\[ \Gamma_\mcV(M) \to M \to L_\mcV(M). \]

Rickard idempotent modules allow us to develop a theory of support
for modules in $\StMod(kG)$ as follows. Let $\fp\subseteq H^*(G,k)$ be 
homogeneous prime ideal. We choose specialisation closed subsets
$\mcV$ and $\mcW$ of $\mcV_G$ with the property that $\mcV\not\subseteq\mcW$
but $\mcV\subseteq\mcW\cup\{\fp\}$ (i.e., $\mcV\smallsetminus\mcW=\{\fp\}$).
Then we write $\Gamma_\fp$ for the functor 
$L_\mcW\Gamma_\mcV\cong\Gamma_\mcV L_\mcW$.
The functor $\Gamma_\fp$ defined in this way is independent of the
choice of $\mcV$ and $\mcW$ with the given properties, see
\cite[Theorem~6.2]{Benson/Iyengar/Krause:2008a}. Note that if $\fp$ is
the maximal prime $\fm$, we have $\Gamma_\fm=0$.

\begin{defn}
If $M$ is an object in $\StMod(kG)$, then the support
$\mcV_G(M)$ is the subset of $\mcV_G$ consisting of
\[ \{\fp\in\mcV_G\mid \Gamma_\fp(M)\ne 0\}. \]
\end{defn}

The following properties of $\mcV_G(M)$ are proved 
in \cite{Benson/Carlson/Rickard:1997a} for $k$ algebraically
closed, and in \cite{Benson/Iyengar/Krause:bik3} for a general field $k$.
This theorem should be compared with 
Theorem~\ref{th:VGM}.

\begin{theorem}
\label{th:cVGM}
Let $M$, $M_1$, $M_2$, $M_3$ be $kG$-modules.
\begin{enumerate}
\item $\mcV_G(M)=\varnothing$ 
if and only if $M$ is a projective module.
\item If $M$ is finitely generated then $\mcV_G(M)$ is equal
to the set of closed homogeneous irreducible subvarieties of $V_G(M)$.
\item For a small family $M_\alpha$ of $kG$-modules we have
\[ 
\mcV_G(\bigoplus_\alpha M_\alpha)=\bigcup_\alpha\mcV_G(M_\alpha). 
\]
\item $\mcV_G(M_1\otimes_k M_2)=\mcV_G(M_1)\cap \mcV_G(M_2)$.
\item If $0\to M_1\to M_2 \to M_3 \to 0$ is a short exact sequence
then for $i=1,2,3$, $\mcV_G(M_i)$ is contained in the union of the
supports of the other two modules.
\item $\mcV_G(\Gamma_\mcV(k))=\mcV\smallsetminus\{\fm\}$,
$\mcV_G(L_\mcV(k))=\mcV_G\smallsetminus\mcV$ and $\mcV_G(\Gamma_\fp(k))=\{\fp\}$ for $\fp\ne\fm$.
\item Given a subset $\mcV\subseteq V_G$, there exists a module $M$ with $\mcV_G(M)=\mcV$. For example we could
take $M=\bigoplus_{\fp\in\mcV}\Gamma_\fp(k)$.
\end{enumerate}
\end{theorem}

The proof of the tensor product theorem 
given in \cite{Benson/Carlson/Rickard:1997a}
depends on comparison with rank varieties and 
the following version of Dade's lemma for modules in $\StMod(kG)$, 
which appeared as Theorem~5.2 of 
\cite{Benson/Carlson/Rickard:1996a}.

\begin{theorem}\label{th:Dade}
Let $E=\langle g_1,\dots,g_r\rangle\cong(\bZ/p)^r$,
let $k$ be an algebraically closed field of characteristic $p$, 
and let $K$ be an algebraically closed
extension field of $k$ of transcendence degree at least $r-1$.
Then a $kE$-module $M$ is projective if and only if for every choice
of $(\lambda_1,\dots,\lambda_r)\in K^r$ with not all the $\lambda_i$ equal
to zero, the restriction
of $K\otimes_k M$ to the cyclic subgroup of $KE$ of order $p$ generated by
\[ 1+\lambda_1(g_1-1)+\dots+\lambda_r(g_r-1) \]
is projective.
\end{theorem}

The role of this version of Dade's lemma in the development
given in \cite{Benson/Carlson/Rickard:1996a,Benson/Carlson/Rickard:1997a}
is what necessitates the requirement that $k$ is algebraically closed.
Later we shall describe another proof avoiding rank varieties
and avoiding any version of Dade's lemma, and which works for a
general field $k$.

\begin{remark}[Maximal versus prime ideals]
The reader will have observed that when we were dealing with 
the stable category $\stmod(kG)$ of finitely generated modules, 
we used $V_G$, the spectrum of maximal ideals in
$H^*(G,k)$, whereas for the stable category $\StMod(kG)$ of all modules
we used $\mcV_G$, the spectrum of homogeneous prime ideals in $H^*(G,k)$.
In the case of a finitely generated module, $\mcV_G(M)$
is determined by $V_G(M)$, see Theorem \ref{th:cVGM}~(ii), whereas
for an infinitely generated module it is not, see part (vii) of
the same theorem.

It would have been possible to use $\mcV_G$ consistently throughout,
but we chose not to, partly for historical reasons.
The origin of the use of $V_G$ is Quillen's work 
\cite{Quillen:1971b,Quillen:1971c}, and much of the literature on
finitely generated modules has been written in this context.

A theorem of Hilbert states that for the maximal ideal spectrum
and the homogeneous prime ideal spectrum of a finitely generated
commutative algebras over a field, 
each determines the other.
A graded commutative ring is commutative modulo its nil
radical, so this applies here.

The homogeneous prime spectrum of a nonstandardly graded ring
can be quite confusing. For example a polynomial ring can give a
spectrum with singularities even though it is regular in the
commutative algebra sense. An explicit example of this is the
ring $k[x,y,z]$ with $|x|=2$, $|y|=4$, $|z|=6$. The affine open
patch corresponding to $y\ne 0$ has coordinate ring generated
by $\alpha=x^2/y$, $\beta=xz/y^2$, $\gamma=z^2/y^3$, with the
single relation $\alpha\gamma=\beta^2$. This has a singularity
at the origin. This example arises as $H^*(G,k)$ modulo its nil 
radical with $G=(\bZ/p)^3\rtimes\Sigma_3$ for $p\ge 5$, and $k$ a
field of characteristic $p$.

Finally, in the Quillen stratification theorem as described in
\cite{Quillen:1971b,Quillen:1971c} (see also Section \ref{se:Qstrat} of
this survey), 
part of the statement
is that $N_G(E)/C_G(E)$ acts freely on the stratum corresponding
to $E$. This is true for inhomogeneous maximal, but not for
homogeneous prime ideals. Fortunately, our use of Quillen stratification
does not involve this feature.
\end{remark}

\section{Classification of tensor ideal thick subcategories}

The key step in the classification of tensor ideal thick subcategories
of $\stmod(kG)$ (Theorem~\ref{th:thick}) is the following theorem.

\begin{theorem}
Let $M$ be a finitely generated $kG$-module with
$k$ algebraically closed, and let $\mcV$ be the
collection of closed homogeneous irreducible subvarieties of
$V_G(M)$. Then the tensor ideal thick subcategory
generated by $M$ is equal to $\fc_\mcV$.
\end{theorem}
\begin{proof}
Write $\fc$ for the tensor ideal thick subcategory generated by $M$.
It is clear that $\fc$ is contained in $\fc_\mcV$.
So the natural transformation $\Gamma_\fc\Gamma_\mcV\to\Gamma_\fc$
is an isomorphism. So if $N$
is any $kG$-module then we obtain a triangle
\[ \Gamma_\fc(N) \to \Gamma_\mcV(N)\to 
L_\fc\Gamma_\mcV(N).\]
The first two terms in this triangle are in $\Loc(\fc_\mcV)$, and hence
so is the third. So
\[ \mcV_G(L_\fc\Gamma_\mcV(N))\subseteq\mcV. \]

If $S$ is any simple $kG$-module then
\[ \sHom_{kG}(S,M^*\otimes_kL_\fc\Gamma_\mcV(N))
\cong \sHom_{kG}(S\otimes_kM,L_\fc\Gamma_\mcV(N))=0. \]
Thus $M^*\otimes_kL_\fc\Gamma_\mcV(N)$ is projective,
and hence
\[ \varnothing=\mcV_G(M^*\otimes_kL_\fc\Gamma_\mcV(N))
=\mcV_G(M)\cap\mcV_G(L_\fc\Gamma_\mcV(N))
=\mcV_G(L_\fc\Gamma_\mcV(N)). \]
It follows that $L_\fc\Gamma_\mcV(N)$ is projective,
and so $\Gamma_\fc(N) \to \Gamma_\mcV(N)$ is
a stable isomorphism. Thus $\fc=\fc_\mcV$.
\end{proof}

\section{Localising subcategories of $\StMod(kG)$}

First we describe the goal. As with the thick subcategories of
$\stmod(kG)$, we only hope to classify the tensor ideal localising
subcategories of $\StMod(kG)$; this is all localising subcategories
only in the case where $G$ is a $p$-group.

There are a lot more tensor ideal localising subcategories of
$\StMod(kG)$ than tensor ideal thick subcategories of $\stmod(kG)$,
as we can see by comparing Theorems~\ref{th:VGM} and \ref{th:cVGM}.
For \emph{any} subset $\mcV$, not just for a specialisation closed one,
let $\sfC_\mcV$ be the full subcategory of $\StMod(kG)$ consisting of
those modules $M$ with $\mcV_G(M)\subseteq\mcV$. The following lemma
is the analogue of Lemma~\ref{le:CV}.

\begin{lemma}
If $\mcV$ is a subset of $\mcV_G$ then $\sfC_\mcV$ is a tensor ideal localising 
subcategory of $\StMod(kG)$.
\end{lemma}

The statement of the classification theorem is as follows.

\begin{theorem}[Benson, Iyengar and Krause \cite{Benson/Iyengar/Krause:bik3}]%
\label{th:localising}
There is a one to one correspondence between tensor ideal localising
subcategories of $\StMod(kG)$ and subsets of $\mcV_G\smallsetminus\{0\}$.
Under this correspondence, $\mcV$ corresponds to $\sfC_\mcV$.
\end{theorem}

The proof of this classification theorem
is much harder than the corresponding proof for 
tensor ideal thick subcategories of $\stmod(kG)$.
If one tries to mimic the arguments of Neeman \cite{Neeman:1992a},
there is a basic obstruction, which is the lack of appropriate
field objects. Specifically, given a prime $\fp\subseteq H^*(G,k)$,
a field object for $\fp$ would be a module $M$ such that
$\hat H^*(G,M)=\sHom^*_{kG}(k,M)$ is isomorphic to the graded field
of fractions of $H^*(G,k)/\fp$. Such an object does not always exist,
as can be seen using the obstruction theory of Benson, Krause 
and Schwede \cite{Benson/Krause/Schwede:2004a,Benson/Krause/Schwede:2005a}.

There is one case where there are field objects, and that is the
case of an elementary abelian $2$-group $E$. The point here is that
the group algebra of an elementary abelian $2$-group is the same as
an exterior algebra in characteristic two. The basic property of
the cohomology in this case is that it is ``formal'', in the sense that
the cochains and the cohomology are equivalent---there is no higher order
information. This is made precise by the Bernstein--Gelfand--Gelfand 
(BGG) correspondence \cite{Bernstein/Gelfand/Gelfand:1978a}, which
gives a correspondence between appropriate categories of modules
for the exterior algebra and for a polynomial algebra, where the
polynomial algebra is regarded as the cohomology of the exterior algebra
and \emph{vice versa}. In this case, one can perform the classification
for the graded polynomial ring $H^*(E,k)$ and then use 
the BGG correspondence to prove the classification for $\StMod(kE)$; see 
\cite{Benson/Iyengar/Krause:bik6} for details.

For a general finite group in characteristic two, the classification
can be achieved using the Quillen stratification theorem and
Chouinard's theorem to reduce to elementary abelian subgroups.
We shall describe these later.

If $E$ is an elementary abelian $p$-group for $p$ odd, the problem
is that the cochains are not formal. However, there is a device
coming from commutative algebra that allows us to reduce to the
formal case. Namely, we regard the group algebra $kE$ as a 
\emph{complete intersection}, and then use a suitable Koszul complex.
We shall describe this construction later.

Before addressing these points, we wish to set up a slightly cleaner
version of the module category, where the origin $\{0\}\in\mcV_G$
is not excluded, as it is in Theorems~\ref{th:thick} and \ref{th:localising}.
This is the category $\KInj(kG)$ described in the next section. It
bears the same relation to $\StMod(kG)$ as the bounded derived
category $\sfD^b(\mmod(kG))$ does to $\stmod(kG)$. 
It might be thought that $\sfD(\Mod(kG))$ would play
this role, but the compact objects in $\sfD(\Mod(kG))$ are just the
\emph{perfect complexes}, namely the complexes quasiisomorphic to
finite complexes of finitely generated \emph{projective} modules.
So we would get $\sfK^b(\proj(kG))$ rather than the desired $\sfD^b(\mmod(kG))$.

\section{The category $\KInj(kG)$}

The objects of $\KInj(kG)$ are the complexes of injective 
(or equivalently projective, see Lemma~\ref{le:pr=inj})
$kG$-modules. We should emphasise that this means \emph{unbounded} complexes
of not necessarily finitely generated injective modules.
The arrows are homotopy classes of degree preserving maps of complexes.
This is a triangulated category in which the triangles come from the
mapping cone construction. This category is investigated in detail in
Krause \cite{Krause:2005a}, Benson and Krause \cite{Benson/Krause:2008a}.

Let $\KacInj(kG)$ be the full subcategory of $\KInj(kG)$ whose objects
are the acyclic complexes. These objects can be described as 
\emph{Tate resolutions} of modules.

\begin{defn}
If $M$ is a $kG$-module then a \emph{Tate resolution} of $M$ is formed
by splicing together an injective resolution and a projective resolution
of $M$:
\[ \xymatrix@=5mm{\cdots\ar[r]&P_1\ar[r]&P_0\ar[rr]\ar[dr]&&I_0\ar[r]&
I_1\ar[r]&\cdots\\
&&&M\ar[ur]\ar[dr]\\&&0\ar[ur]&&0} \]
\end{defn}

Every acyclic complex of injectives 
\[ \cdots \to P_1\to P_0 \to P_{-1} \to P_{-2} \to \cdots \]
can be regarded as a Tate resolution of the image of $P_0\to P_{-1}$.
Furthermore, given a module homomorphism $M\to M'$, it can be extended
to a map of Tate resolutions. Such extensions are homotopic if and
only if the homomorphisms differ by an element of $\PHom_{kG}(M,M')$.
It follows that Tate resolutions give an equivalence of categories
between $\StMod(kG)$ and $\KacInj(kG)$.

We write $tk$ for a Tate resolution of the trivial module $k$, as
an object in $\KInj(kG)$, $ik$ for an injective resolution, and
$pk$ for a projective resolution. So there is a triangle in $\KInj(kG)$
of the form
\[ pk \to ik \to tk \]
expressing $tk$ as the mapping cone of the map $pk\to ik$. 
We can make projective, injective and Tate resolutions of any module
by tensoring:
\[ M \otimes_k pk  \to M\otimes_k ik \to M \otimes_k tk. \]

Tensor products in $\KInj(kG)$ are taken to be tensor products over $k$ of
complexes of modules. Note that $ik$ is the tensor identity of $\KInj(kG)$.

Now $\KacInj(kG)$ is a localising subcategory of $\KInj(kG)$, and the
quotient category $\KInj(kG)/\KacInj(kG)$ 
is the unbounded derived category $\sfD(\Mod(kG))$. The inclusion 
\[ \KacInj(kG)\to\KInj(kG) \] 
and the quotient functor 
\[ \KInj(kG)\to\sfD(\Mod(kG)) \]
each have both a left and a right adjoint, so that we get a 
diagram of categories and functors
\[ \StMod(kG) \simeq \KacInj(kG) \ \begin{smallmatrix} \Hom_k(tk,-) \\
\hbox to 50pt{\leftarrowfill} \\ \hbox to 50pt{\rightarrowfill} \\ 
\hbox to 50pt{\leftarrowfill} \\ - \otimes_k tk
\end{smallmatrix} \ \KInj(kG) \ \begin{smallmatrix} \Hom_k(pk,-) \\
\hbox to 50pt{\leftarrowfill} \\ \hbox to 50pt{\rightarrowfill} \\ 
\hbox to 50pt{\leftarrowfill} \\ - \otimes_k pk
\end{smallmatrix} \ \sfD(\Mod(kG)). \]

The compact objects in these categories are only preserved by
the left adjoints, and give us back Rickard's sequence
\[ \stmod(kG) \leftarrow \sfD^b(\mmod(kG)) \leftarrow \sfK^b(\proj(kG)). \]

\begin{lemma}
The only tensor ideal localising subcategories of the unbounded derived category $\sfD(\Mod(kG))$
are zero and the whole category.
\end{lemma}
\begin{proof}
  Let $\sfC$ be a non-zero tensor ideal localising subcategory. If $X$
  is a non-zero object in $\sfC$ then $X$ has non-zero homology. The
  homology of $kG\otimes_k X$ is thus non-zero and free. Any summand
  isomorphic to $kG$ of the homology splits off as a summand of
  $kG\otimes_k X$, and so $kG$ is in $\sfC$. But $kG$ generates the
  whole of $\sfD(\Mod(kG))$.
\end{proof}

It follows from this lemma that corresponding to any 
tensor ideal localising subcategory of $\StMod(kG)$, there are two
tensor ideal localising subcategories of $\KInj(kG)$, one of which
is contained in $\KacInj(kG)$ and the other of which is generated by
this together with the image of $\sfD(\Mod(kG))$ under $-\otimes_k pk$.

Correspondingly, we have a notion of support varieties for objects
in $\KInj(kG)$. The definitions involve exactly the same definitions
of functors $\Gamma_\fc$ and $L_\fc$ as for $\StMod(kG)$, and the
only difference is that if $X$ is an object in $\KInj(kG)$ then 
the origin $\{0\}$ is either in $\mcV_G(X)$ or not according as
the complex $X$ has homology or not.

This allows us to formulate the following version of the classification
theorem for $\KInj(kG)$, where the origin has lost its special role.

\begin{theorem}\label{th:loc-KInj}
There is a one to one correspondence between tensor ideal localising
subcategories of $\KInj(kG)$ and subsets of $\mcV_G$.
Under this correspondence, $\mcV$ corresponds to $\sfC_\mcV$.
\end{theorem}

This is the theorem whose proof we shall outline in the sections to follow.

\section{Support for triangulated categories}

At this stage, it is appropriate to describe the general setup
introduced in \cite{Benson/Iyengar/Krause:2008a} for discussing
support for objects in triangulated categories. This setup allows us
to move from one category to another without too much effort. Then the
game plan, which is inspired by the work of Avramov, Buchweitz,
Iyengar, and Miller~\cite{Avramov/Buchweitz/Iyengar/Miller:2010a}, is
as follows:
\begin{enumerate}
\item Reduce from a finite group to its elementary abelian subgroups.
\item Use a Koszul construction to move from complexes of modules over an
  elementary abelian group to differential graded modules over a
  graded exterior algebra.
\item Use a version of the BGG correspondence to move from a graded
  exterior algebra to a graded polynomial algebra.
\item Use a version of Neeman's classification \cite{Neeman:1992a} to
  deal directly with differential graded modules over graded
  polynomial algebras.
\end{enumerate}

\begin{defn}
  Let $\sfT$ be a compactly generated triangulated category with small
  coproducts. We write $\Sigma$ for the shift in $\sfT$ and $\sfT^c$
  for the full subcategory of compact objects in $\sfT$.

We write $Z^*(\sfT)$ for the
\emph{graded centre} of $\sfT$. Namely, $Z^*(\sfT)$ is the graded 
ring whose degree $n$ component $Z^n(\sfT)$ 
is the set of natural transformations
\[ \eta\colon\Id_\sfT\to\Sigma^n \] 
satisfying 
\[ \eta\Sigma=(-1)^n\Sigma\eta. \]
Then $Z^*(\sfT)$ is a 
\emph{graded commutative ring}, in the
sense that for $x,y\in Z^*(\sfT)$ we have
\[ yx=(-1)^{|x||y|}xy. \] 

Let $R$ be a graded commutative Noetherian ring.
We say that $\sfT$ is an \emph{$R$-linear triangulated category} if
we are given a homomorphism of graded commutative rings 
$\phi\colon R\to Z^*(\sfT)$. 
This amounts to giving, for each object $X$ in $\sfT$, a homomorphism
of graded rings 
\[ 
\phi_X\colon R\to \End_\sfT^*(X), 
\]
such that for $X$ and $Y$ objects in $\sfT$,  the two induced actions of an element $r\in R$ on 
$\alpha\in\Hom_\sfT^*(X,Y)$ via $\phi_{X}$ and $\phi_{Y}$ are related by
\[ 
\phi_Y(r)\alpha=(-1)^{|r||\alpha|}\alpha\phi_X(r). 
\]
\end{defn}

\begin{eg}
If $\sfT=\StMod(kG)$ then $\sfT^c=\stmod(kG)$ and $\Sigma=\Omega^{-1}$.
In this case, we take $R=H^*(G,k)$. Although it might seem more 
natural to use Tate cohomology $\hat H^*(G,k)=\sHom^*_{kG}(k,k)$,
the problem is that this ring is usually not Noetherian. This
is related to the awkward role of the origin in the theory of 
the support varieties for $\StMod(kG)$.

If $\sfT=\KInj(kG)$ then $\sfT^c=\sfD^b(\mmod(kG))$ and $\Sigma$ is 
the usual shift. In this case, we also take $R=H^*(G,k)$. In this
case, $R$ is just the graded endomorphism ring of the tensor identity $ik$,
and the origin no longer has an awkward role.
\end{eg}

We write $\Spec R$ for the set of \emph{homogeneous} prime ideals in
$R$.  For each prime $\fp\in\Spec R$ and each graded $R$-module $M$,
we denote by $M_\fp$ the graded localisation of $M$ at $\fp$.

If $\mcV$ is a specialisation closed subset of $\Spec R$, we set
\[ \sfT_\mcV = \{X\in \sfT\mid \Hom_\sfT^*(C,X)_\fp=0 \text{ for all } 
C\in\sfT^c,\,\fp\in\Spec R\smallsetminus \mcV\}. \]
This is a localising subcategory of $\sfT$, and there is a
localisation functor $L_\mcV\colon\sfT\to\sfT$ such that $L_\mcV(X)=0$
if and only if $X$ is in $\sfT_\mcV$. We then define $\Gamma_\mcV(X)$
by completing $X\to L_\mcV(X)$ to a triangle:
\[ \Gamma_\mcV(X) \to X \to L_\mcV(X). \]

\begin{eg}\label{eg:StMod}
In the case where $\sfT$ is either 
$\StMod(kG)$ or $\KInj(kG)$, and $R=H^*(G,k)$, 
this is exactly Rickard's triangle for the subset $\mcV$ of $\Spec H^*(G,k)$.
\end{eg}

As in Section~\ref{se:Rickard}, if $\fp\in\Spec R$ we choose
specialisation closed subsets
$\mcV$ and $\mcW$ of $\Spec R$ satisfying $\mcV\not\subseteq\mcW$ and
$\mcV\subseteq\mcW\cup{\fp}$ and define
\[ \Gamma_\fp=L_\mcW\Gamma_\mcV=\Gamma_\mcV L_\mcW. \]
Again, this is independent of choice of $\mcV$ and $\mcW$ with the
given properties.

\begin{defn}
If $X$ is an object in $\sfT$ then the \emph{support} of $X$ is
the subset
\[ \supp_R(X)=\{\fp\in\Spec R \mid \Gamma_\fp(X)\ne 0\}. \]
This is a subset of
\[ \supp_R(\sfT)=\{\fp\in\Spec R\mid \Gamma_\fp(\sfT)\ne 0\}\,
\subseteq \,\Spec R. \]
\end{defn}

\begin{eg}
Continuing Example~\ref{eg:StMod}, this notion of support  agrees
with $\mcV_G(X)$ as defined in Section~\ref{se:Rickard}. We have
\begin{align*} 
\supp_{H^*(G,k)}(\StMod(kG))&=\mcV_G\smallsetminus \{0\}, \\
\supp_{H^*(G,k)}(\KInj(kG))&=\mcV_G. 
\end{align*} 
\end{eg}

The following theorem summarises the properties of support for
$R$-linear triangulated categories; confer Theorems~\ref{th:VGM} and
\ref{th:cVGM}. Proofs can be found in
\cite{Benson/Iyengar/Krause:2008a}.

\begin{theorem}
Let $X$, $Y$ and $Z$ be objects in $\sfT$. Then
\begin{enumerate}
\item $X= 0$ if and only if $\supp_R(X)=\varnothing$.
\item $\supp_R(\Sigma X)=\supp_R(X)$.
\item For a small family of objects $X_\alpha$ we have
\[ \supp_R(\bigoplus_\alpha X_\alpha)=\bigcup_\alpha \supp_R(X_\alpha). \]
\item If $X\to Y \to Z$ is a triangle in $\sfT$ then
\[ \supp_R(Y) \subseteq \supp_R(X) \cup \supp_R(Z). \]
\item For $\mcV\subseteq\Spec R$ we have
\begin{align*}
\supp_R(\Gamma_\mcV(X))&=\mcV\cap\supp_R(X)\\
\supp_R(L_\mcV(X))&=(\Spec R\smallsetminus\mcV)\cap \supp_R(X).
\end{align*}
\item If $\cl(\supp_R(X))\cap \supp_R(Y)=\varnothing$
then $\Hom_\sfT(X,Y)=0$. 
\end{enumerate}
\end{theorem}
Here, $\cl(\mcV)$ denotes the specialisation closure of a subset $\mcV$.

\section{Tensor triangulated categories}
\label{se:tensor}

In this section we consider compactly generated triangulated
categories with some extra structure. Namely, we want to consider an
internal tensor product
\[
\otimes\colon \sfT\times\sfT\to\sfT, 
\] 
exact in each variable, preserving small coproducts, and with a unit $\one$.  It then follows (using the
Brown representability theorem) that there are \emph{function objects}
$\fHom(X,Y)$ in $\sfT$ with natural isomorphisms
\[ \Hom_\sfT(X\otimes Y,Z)\cong \Hom_\sfT(X,\fHom(Y,Z)). \]
We define the \emph{dual} of an object $X$ in $\sfT$ to be
\[ X^\vee=\fHom(X,\one). \]

\begin{defn}
We say that $(\sfT,\otimes,\one)$ is 
a \emph{tensor triangulated category} if the following hold:
\begin{enumerate}
\item $\sfT$ is a compactly generated triangulated category with
small coproducts.
\item The tensor product $\otimes$ and unit $\one$ make $\sfT$ 
a symmetric monoidal category.
\item The tensor product is exact in each variable and preserves
small coproducts.
\item The unit $\one$ is compact.
\item Compact objects are \emph{strongly dualisable} in the
sense that if $C$ is compact and $X$ is any object in $\sfT$ then
the canonical map
\[ C^\vee\otimes X \to \fHom(C,X) \]
is an isomorphism.
\end{enumerate}
\end{defn}
See \cite{Benson/Iyengar/Krause:2008a,Benson/Iyengar/Krause:bik4} for further discussion of this structure. 
 
\begin{eg}
Both $\StMod(kG)$ and $\KInj(kG)$ are tensor triangulated categories.
\end{eg}

The symmetric monoidal structure ensures that the endomorphism ring of the tensor identity $\End^*_\sfT(\one)$ is a graded commutative ring. Given any object $X$, this ring acts on $\End^*_\sfT(X)$ via
\[ 
\End^*_\sfT(\one)\xrightarrow{X\otimes -}\End^*_\sfT(X). 
\]
So if $R$ is a graded commutative Noetherian  ring, a homomorphism  $R\to\End^*_\sfT(\one)$ gives $\sfT$ a structure of an $R$-linear category. 

\begin{defn}
We say that the action of $R$ on $\sfT$ is
\emph{canonical} if it arises from a homomorphism $R\to\End^*_\sfT(\one)$
as described in the previous paragraph.
\end{defn}

\begin{prop}
If the action of $R$ on $\sfT$ is canonical then for each specialisation
closed subset $\mcV\subseteq\Spec R$ and each prime $\fp\in\Spec R$
there are natural isomorphisms
\[ \Gamma_\mcV(X)\cong X\otimes\Gamma_\mcV(\one),\qquad
L_\mcV(X)\cong X \otimes L_\mcV(\one),\qquad
\Gamma_\fp(X)\cong X \otimes \Gamma_\fp(\one). \]
\end{prop}

\section{The local-global principle}
\label{se:loc-glob}

If $X$ is an object or collection of objects in a triangulated category $\sfT$, we write $\Loc_\sfT(X)$ for the localising subcategory of $\sfT$ generated by $X$. 

\begin{defn}
Let $\sfT$ be an $R$-linear triangulated category with small coproducts. We say that the \emph{local-global principle holds} for $\sfT$ if for each object $X$ in $\sfT$ we have
\[ 
\Loc_\sfT(X)=\Loc_\sfT(\{\Gamma_\fp(X)\mid \fp\in\supp_R(\sfT)\}). 
\]
\end{defn}

This condition is usually satisfied, because of the following theorem
\cite[Corollary~3.5]{Benson/Iyengar/Krause:bik2}:

\begin{theorem}
The local-global principle holds provided $R$ has finite Krull dimension.
\end{theorem}

For a tensor triangulated category $\sfT$, we modify the definition slightly.
If $X$ is an object or collection of objects in $\sfT$, we write 
$\Loc^\otimes_\sfT(X)$ for the tensor ideal localising subcategory
of $\sfT$ generated by $X$. The following theorem says that
for a tensor triangulated category the appropriate analogue of
the local-global principle always holds
\cite[Theorem~7.2]{Benson/Iyengar/Krause:bik2}:

\begin{theorem}\label{th:BIK2:7.2}
Let $\sfT$ be an $R$-linear 
tensor triangulated category with canonical $R$-action. 
Then for each object $X$ in $\sfT$ the following holds:
\[ \Loc^\otimes_\sfT(X)=\Loc^\otimes_\sfT(\{\Gamma_\fp(X)\mid\fp\in\Spec R\}). \]
\end{theorem}

When the local-global principle holds, the classification of localising
subcategories can be achieved one prime at a time. The way to express this
is via the following maps:
\[ \left\{\begin{gathered}
\text{Localising}\\ \text{subcategories of $\sfT$}
\end{gathered}\; \right\} 
\xymatrix@C=3pc {\ar@<1ex>[r]^-{{\sigma}} & \ar@<1ex>[l]^-{{\tau}}}
\left\{
\begin{gathered}
  \text{Families $(\sfS(\fp))_{\fp\in\supp_R(\sfT)}$ with $\sfS(\fp)$ a}\\
  \text{localising subcategory of $\Gamma_\fp(\sfT)$}
\end{gathered}\;\right\} \] 
where $\sigma(\sfS)=(\sfS\cap\Gamma_\fp(\sfT))$ and 
$\tau(\sfS(\fp))=\Loc_\sfT(\sfS(\fp)\mid\fp\in\supp_R(\sfT))$. The following
is \cite[Proposition~3.6]{Benson/Iyengar/Krause:bik2}.

\begin{theorem}\label{th:BIK2:3.6}
If the local-global principle holds then the maps $\sigma$ and
$\tau$ are mutually inverse bijections.
\end{theorem}

In other words, specifying a localising subcategory of $\sfT$ is
equivalent to specifying a localising subcategory of $\Gamma_\fp(\sfT)$
for each prime $\fp\in\supp_R(\sfT)$.

Since the tensor ideal version of the local-global principle always
holds in a tensor triangulated category, we have the following theorem.

\begin{theorem}
Let $\sfT$ be an $R$-linear tensor triangulated category with
canonical $R$-action. Then the maps $\sigma$ and $\tau$ above 
induce mutually 
inverse bijections between the 
tensor ideal localising subcategories of $\sfT$
and families $(\sfS(\fp))_{\fp\in\supp_R(\sfT)}$ with $\sfS(\fp)$ a
tensor ideal localising subcategory of $\Gamma_\fp(\sfT)$.
\end{theorem}

\begin{remark}
If the tensor identity $\one$ generates $\sfT$ then every
localising subcategory of $\sfT$ is tensor ideal. In the case
of $\StMod(kG)$ and $\KInj(kG)$ this holds if and only
if $G$ is a finite $p$-group.
\end{remark}

\section{Stratifying triangulated categories}

In the last section, we showed how to classify localising subcategories
of an $R$-linear triangulated category $\sfT$ one prime at a time. The
easiest case is where each $\Gamma_\fp(\sfT)$ with $\fp\in\supp_R(\sfT)$,
is a minimal subcategory.

\begin{defn}
We say that a localising subcategory of $\sfT$ is \emph{minimal} if
it is non-zero and has no proper non-zero localising subcategories.

We say that $\sfT$ is \emph{stratified} by $R$ if the local-global
principle holds, and each $\Gamma_\fp(\sfT)$ with $\fp\in\supp_R\sfT$ 
is a minimal localising subcategory.
\end{defn}

Suppose that $\sfT$ is stratified by $R$, and consider the maps
$\sigma$ and $\tau$ of Section~\ref{se:loc-glob}. 
To name a family 
$(\sfS(\fp))_{\fp\in\supp_R(\sfT)}$ with $\sfS(\fp)$ a localising
subcategory of $\Gamma_\fp(\sfT)$, we just have to name the
set of $\fp\in\supp_R(\sfT)$ for which $\sfS(\fp)=\Gamma_\fp(\sfT)$, since the
remaining primes $\fp$ will satisfy $\sfS(\fp)=0$.
So the following is a direct consequence of Theorem~\ref{th:BIK2:3.6}.

\begin{theorem}
Let $\sfT$ be an $R$-linear triangulated category. If $\sfT$ is
stratified by $R$ then the maps $\sigma$ and $\tau$  
establish a bijection between the
localising subcategories of $\sfT$ and the subsets of $\supp_R(\sfT)$.
\end{theorem}

In the case of an $R$-linear tensor triangulated category $\sfT$, because the
corresponding local-global principle is automatic (Theorem~\ref{th:BIK2:7.2}),
we say that $\sfT$ is \emph{stratified} by $R$ if for each 
$\fp\in\supp_R(\sfT)$,
$\Gamma_\fp(\sfT)$ is minimal as a tensor ideal localising subcategory.

\begin{theorem}
Let $\sfT$ be an $R$-linear tensor triangulated category. If $\sfT$ is
stratified by $R$ then the maps $\sigma$ and $\tau$  
establish a bijection between the tensor ideal
localising subcategories of $\sfT$ and the subsets of $\supp_R(\sfT)$.
\end{theorem}

In order to prove stratification in a given situation, we need
a criterion for minimality of localising subcategories. The
following can be found in Lemma~4.1 of \cite{Benson/Iyengar/Krause:bik2}
and Lemma~3.9 of \cite{Benson/Iyengar/Krause:bik3}.

\begin{lemma}\label{le:test}
\begin{enumerate}
\item  Let $\sfT$ be an $R$-linear triangulated category.
A non-zero localising subcategory $\sfS$ of $\sfT$ is minimal if and
only if for all non-zero objects $X$ and $Y$ in $\sfS$ we have
$\Hom^*_\sfT(X,Y)\ne 0$.

\item Let $\sfT$ be an $R$-linear tensor triangulated category.
A non-zero tensor ideal localising subcategory $\sfS$ of 
$\sfT$ is minimal if and only if for all non-zero objects $X$ and $Y$
in $\sfS$ there exists an object $Z$ in $\sfT$ such that 
$\Hom^*_\sfT(X\otimes Z,Y)\ne 0$. The object $Z$ may be taken to be
a compact generator for $\sfT$.
\end{enumerate}
\end{lemma}

We can now restate Theorems~\ref{th:localising} and \ref{th:loc-KInj}.

\begin{theorem}\label{th:strat}
The tensor triangulated categories $\StMod(kG)$ and
$\KInj(kG)$ are stratified by $H^*(G,k)$.
\end{theorem}

In the next few sections, we outline the proof of this theorem.
A lot of details will be swept under the carpet here, but are spelt out
in \cite{Benson/Iyengar/Krause:bik3}.

\section{Graded polynomial algebras}

The first step in proving Theorem~\ref{th:strat} is to stratify the
derived category of differential graded modules a polynomial ring. Let
$S$ be a graded polynomial algebra over the field $k$ on a finite
number of indeterminates. If $k$ does not have characteristic two, we
assume that the degrees of the indeterminates are even, so that $S$ is
graded commutative.

We view $S$ as a \emph{differential graded}
(abbreviated to dg) algebra with zero differential,
and we write $\sfD(S)$ for the derived category of dg $S$-modules. 
The objects of this category are the dg $S$-modules, 
and the morphisms are homotopy classes of degree preserving
chain maps with quasi-isomorphisms inverted. See for example 
Keller \cite{Keller:1994a} for further details.

The category $\sfD(S)$ is a tensor triangulated category in which the
tensor product is the derived tensor product over $S$, $X\otimes^\bL_S
Y$ and the tensor identity is $S$, which is a compact generator for
$\sfD(S)$. In particular, every localising subcategory is tensor
ideal. The canonical action of $S$ on $\sfD(S)$ makes $\sfD(S)$ into
an $S$-linear tensor triangulated category. The following theorem is a
dg analogue of the theorem of Neeman \cite{Neeman:1992a}, and is
proved in Theorem~5.2 of \cite{Benson/Iyengar/Krause:bik3}. The
existence of field objects plays a crucial role in the proof.

\begin{theorem}
The category $\sfD(S)$ is stratified by the canonical $S$-action. So the
maps $\sigma$ and $\tau$ of Section~\ref{se:loc-glob} are mutually
inverse bijections between the localising subcategories of $\sfD(S)$
and subsets of $\Spec S$.
\end{theorem}

\section{A BGG correspondence}

The second step in the proof of Theorem~\ref{th:strat} is to use
a version of the BGG correspondence to transfer the stratification
from  polynomial rings to exterior algebras.

Let $k$ be a field and let $\Lambda$ be an exterior algebra
on indeterminates $\xi_1,\dots,\xi_c$ of negative odd degrees.
We view $\Lambda$ as a dg algebra with zero differential. 
Let $S$ be a graded polynomial algebra on variables
$x_1,\dots,x_c$ with $|x_i|=-|\xi_i|+1$. Let $J$ be the
dg $\Lambda\otimes_k S$-module with
\[ J^\natural = \Hom_k(\Lambda,k)\otimes_k S \]
and 
\[ d(f\otimes s)=\sum_i \xi_i f \otimes x_i s. \]
If $M$ is a dg $\Lambda$-module then $\Hom_\Lambda(J,M)$
is in a natural way a dg $S$-module.

In general, for a dg algebra $A$, we write $A^\natural$ for the
underlying graded algebra of $A$, forgetting the differential.
If $M$ is a dg $A$-module, we write $M^\natural$ for the
underlying graded $A^\natural$-module. We say that a dg $A$-module $I$ is
\emph{graded-injective} if $I^\natural$ is injective in the category of
graded $A^\natural$-module. Finally, we write $\KInj(A)$ for the
homotopy category of graded-injective dg $A$-modules.

The following version of the BGG correspondence
\cite{Bernstein/Gelfand/Gelfand:1978a} extends the one from
\cite{Avramov/Buchweitz/Iyengar/Miller:2010a}.

\begin{theorem}\label{th:BGG}
The functor 
\[ \Hom_\Lambda(J,-) \colon \KInj(\Lambda)\to \sfD(S) \]
is an equivalence of categories.
\end{theorem}

We give $\Lambda$ a comultiplication
\[ \Delta(\xi_i)=\xi_i\otimes 1+1\otimes \xi_i. \]
This allows us to define a tensor product of dg $\Lambda$-modules,
and this makes $\KInj(\Lambda)$ into a tensor triangulated category.
Its tensor identity is $ik$, which is a compact generator.
We have
\[ S \cong \Ext^*_\Lambda(k,k) \cong \Hom_{\KInj(\Lambda)}(J,J) \]
which makes $\KInj(\Lambda)$ into an $S$-linear tensor triangulated
category with canonical $S$-action. The equivalence described in
Theorem~\ref{th:BGG} allows us to stratify $\KInj(\Lambda)$.

\begin{theorem}\label{th:strat-Lambda}
The category $\KInj(\Lambda)$ is stratified by
the canonical action of $S$. So the maps $\sigma$ and $\tau$
of Section~\ref{se:loc-glob} give a bijection between 
localising subcategories of $\KInj(\Lambda)$ and subsets of
$\Spec S$.
\end{theorem}

\section{The Koszul construction}

If $E$ is an elementary abelian $2$-group and $k$ has characteristic
two, then $kE$ is an exterior algebra on generators of degree zero. So
we can use Theorem~\ref{th:strat-Lambda} to stratify $\KInj(kE)$.
This proves Theorem~\ref{th:strat} in this case; see
\cite{Benson/Iyengar/Krause:bik6} for details. Elementary abelian
$p$-groups with $p$ odd cannot be treated this way. In this section,
we show how to use a Koszul construction to deal with this case.

So let $p$ be a prime, $k$ be a field of characteristic $p$, and 
\[ E=\langle g_1,\dots,g_r\rangle\cong (\bZ/p)^r \]
be an elementary abelian $p$-group. Let $z_i=g_i-1\in kE$, so that
\[ kE=k[z_1,\dots,z_r]/(z_1^p,\dots,z_r^p). \]

We regard $kE$ as a complete intersection, and form the Koszul construction
as follows. Let $A$ be the dg algebra
\[ A=kE\langle y_1,\dots,y_r\rangle \]
where $kE$ is in degree zero, and
$y_1,\dots,y_r$ are exterior generators of degree $-1$ with
\[ y_i^2=0,\qquad y_iy_j=-y_jy_i,\qquad d(y_i)=z_i,\qquad d(z_i)=0. \]

Let $\Lambda$ be an exterior algebra on generators $\xi_1,\dots,\xi_r$
of degree $-1$, regarded as a dg algebra with zero differential,
and $S$ be a polynomial algebra on generators $x_1,\dots,x_r$ of
degree $2$, again regarded as dg algebra with zero differential.
The following is Lemma~7.1 of \cite{Benson/Iyengar/Krause:bik3}.

\begin{lemma}
The map $\phi\colon\Lambda\to A$ defined by 
\[ \phi(\xi_i)=z_i^{p-1}y_i \]
is a quasi-isomorphism of dg algebras. In particular,
\[ \Ext^*_A(k,k)\cong \Ext^*_\Lambda(k,k)\cong S. \]
\end{lemma}

We give $A$ a comultiplication
\[ \Delta(z_i)=z_i\otimes 1 + 1 \otimes z_i,\qquad 
\Delta(y_i)=y_i\otimes 1 + 1 \otimes y_i. \]
This gives $\KInj(A)$ the structure of a tensor triangulated category
with a canonical action of $S$. The following is a special case of 
Proposition~4.6 of \cite{Benson/Iyengar/Krause:bik3}.

\begin{theorem}
The map $\phi\colon\Lambda\to A$ induces an equivalence of 
tensor triangulated categories
\[ \Hom_\Lambda(A,-)\colon \KInj(\Lambda)\to\KInj(A). \]
\end{theorem}

As a consequence, we can stratify $\KInj(A)$. The following is
Theorem~7.2 of \cite{Benson/Iyengar/Krause:bik3}.

\begin{theorem}\label{th:strat-A}
The $S$-linear tensor triangulated category $\KInj(A)$ is stratified 
by the canonical action of $S$. So the maps $\sigma$ and $\tau$ of
Section~\ref{se:loc-glob} give a bijection between localising subcategories
of $\KInj(A)$ and subsets of $\Spec S$.
\end{theorem}

Now the inclusion map $kE\to A$ induces a restriction map
\[ \Ext^*_A(k,k)\to \Ext^*_{kE}(k,k). \]
The structure of $\Ext^*_{kE}(k,k)$ is as follows. 
If $p$ is odd then it is a tensor product of an exterior algebra and
a polynomial algebra
\[ \Ext^*_{kE}(k,k) \cong \Lambda(u_1,\dots,u_r)\otimes 
k[x_1,\dots,x_r] \]
with $|u_i|=1$ and $|x_i|=2$. The elements $x_i$ are the restrictions
of the elements of the same name in $\Ext^*_A(k,k)\cong S$.
If $p=2$ then
\[ \Ext^*_{kE}(k,k)=k[u_1,\dots,u_r] \] 
with $|u_i|=1$. The elements $x_i\in\Ext^*_A(k,k)$ restrict to $u_i^2$.

In both cases, this allows us to regard $S\cong\Ext^*_A(k,k)$ 
as a subring of $H^*(E,k)$
over which it is finitely generated as a module. 
The restriction map from $\Ext^*_A(k,k)$ to
$H^*(E,k)=\Ext^*_{kE}(k,k)$ induces a bijection
\[ \Spec H^*(E,k) \to \Spec S. \]

Using the induction and restriction functors,
the criterion of Lemma~\ref{le:test} 
is essential in deducing
the following theorem from Theorem~\ref{th:strat-A}.
See Theorem~4.11 of \cite{Benson/Iyengar/Krause:bik3} for details.

\begin{theorem}\label{th:strat-kE}
The tensor triangulated category $\KInj(kE)$ is stratified
by the canonical action of $H^*(E,k)$, or equivalently
of $S\subseteq H^*(E,k)$. So the maps $\sigma$ and $\tau$ of
Section~\ref{se:loc-glob} give a bijection between localising
subcategories of $\KInj(kE)$ and subsets of $\Spec H^*(E,k)$.
\end{theorem}

There is an issue here with the tensor structure. The usual diagonal
map $\Delta(g_i)=g_i\otimes g_i$ 
on $kE$ is not compatible with the inclusion $kE\to A$.
There is another diagonal map on $kE$ given by 
$\Delta(z_i)=z_i\otimes 1 + 1\otimes z_i$ coming from regarding
$kE$ as a restricted universal enveloping algebra, and this 
diagonal map is compatible with the inclusion.
Each diagonal map gives rise to a canonical action of 
$H^*(E,k)$ on $\KInj(kE)$, but the two actions are not the same.
Lemma~3.10 of \cite{Benson/Iyengar/Krause:bik3} helps us
out at this point:

\begin{lemma}\label{le:two-actions}
Let $\sfT$ be a triangulated category admitting two tensor triangulated
structures with the same unit $\one$, and assume that $\one$ generates $\sfT$.
Let
\[ \phi,\phi'\colon R\to Z^*(\sfT) \] 
be two actions of $R$ on $\sfT$.
If the maps $R\to\End^*_\sfT(\one)$ induced by $\phi$ and $\phi'$ agree
then $\sfT$ is stratified by $R$ through $\phi$ if and only if it is
stratified by $R$ through $\phi'$.
\end{lemma}

\section{Quillen stratification}\label{se:Qstrat}

There is a general machine for understanding cohomological properties
of a general finite group from its elementary abelian subgroups,
called Quillen stratification. In this section, we explain how
to use this machine to compare localising subcategories of
$\KInj(kE)$ and $\KInj(kG)$.

Let $\mcV_G=\Spec H^*(G,k)$. If $H$ is a subgroup of $G$ then the
restriction map $H^*(G,k)\to H^*(H,k)$ induces a map 
\[ \res^*_{G,H}\colon\mcV_H\to \mcV_G. \]

Quillen \cite{Quillen:1971b,Quillen:1971c} proved the following. 
Given a prime $\fp\in\mcV_G$ 
we say that $\fp$ \emph{originates} in an elementary abelian 
$p$-subgroup $E\le G$ if $\fp$ is in the image of
$\res^*_{G,E}$ but not of $\res^*_{G,E'}$ for $E'$ a proper subgroup
of $E$. 

\begin{theorem}\label{th:Quillen}
Given a prime $\fp\in\mcV_G$, there exists a pair $(E,\fq)$ such
that $\fp$ originates in $E$ and $\res^*_{G,E}(\fq)=\fp$, and all
such pairs are $G$-conjugate. This sets up a bijection between primes
$\fp\in\mcV_G$ and conjugacy classes of pairs $(E,\fq)$ where
$\fq\in\mcV_E$ originates in $E$.
\end{theorem}

In order to be able to use this, we first need a version of the
``subgroup theorem'' for elementary abelian $p$-groups. The following
is Theorem~9.5 of \cite{Benson/Iyengar/Krause:bik3}, and its proof
is a fairly straightforward consequence of the Stratification 
Theorem~\ref{th:strat-kE} for elementary abelian $p$-groups.

\begin{theorem}
\label{th:subgroup-E}
Let $E'\le E$ be elementary abelian $p$-groups. If $X$ is an object
in $\KInj(kE)$ then
\[ 
\mcV_{E'}(X{\downarrow_{E'}})=(\res^*_{E,E'})^{-1}\mcV_E(X). 
\]
\end{theorem}

Next, we need
a version of Chouinard's theorem
\cite{Chouinard:1976a} for $\KInj(kG)$, see Proposition~9.6
of \cite{Benson/Iyengar/Krause:bik3}:

\begin{theorem}\label{th:Chouinard}
An object $X$ in $\KInj(kG)$ is zero if and only if the restriction
of $X$ to every elementary abelian $p$-subgroup of $G$ is zero.
\end{theorem}

We are now ready to outline the proof of Theorem~\ref{th:strat}.
This amounts to showing that for $\fp\in\mcV_G$, the tensor
ideal localising subcategory $\Gamma_\fp\KInj(kG)$ is minimal.
For this purpose, we use the criterion of Lemma~\ref{le:test}.
Let $X$ and $Y$ be non-zero objects in $\Gamma_\fp\KInj(kG)$.
By Theorem~\ref{th:Chouinard} there exists an elementary abelian
subgroup $E_0$ such that $X{\downarrow_{E_0}}$ is non-zero.
Choose a prime $\fq_0\in\mcV_{E_0}(X{\downarrow_{E_0}})$. 
Using standard properties of support under induction and restriction,
we obtain $\res^*_{G,E_0}(\fq_0)=\fp$. So we can choose a pair
$(E,\fq)$ with $E_0\ge E$ and $\fq_0=\res^*_{E_0,E}(\fq)$, so
that the conjugacy class of $(E,\fq)$ corresponds to $\fp$ under
the bijection of Theorem~\ref{th:Quillen}. By Theorem~\ref{th:subgroup-E}
we have $\Gamma_\fq(X{\downarrow_E})\ne 0$. Since the pair $(E,\fq)$
is determined up to conjugacy by $\fp$, the object $Y\in
\Gamma_\fp\KInj(kG)$ also has $\Gamma_\fq(Y{\downarrow_E})\ne 0$.

Let $Z$ be an injective resolution of the permutation module $k(G/E)$,
as an object in $\KInj(kG)$. This has the property that 
\[ X \otimes_k Z \cong X \otimes_k k(G/E)\cong X{\downarrow_E}{\uparrow^G}. \]
Thus using Frobenius reciprocity we have 
\[ \Hom^*_{kG}(X\otimes_k Z,Y)\cong \Hom^*_{kG}(X{\downarrow_E}{\uparrow^G},Y)
\cong \Hom^*_{kE}(X{\downarrow_E},Y{\downarrow_E}). \]
Since $\Gamma_\fq\KInj(kE)$ is minimal,
using the first part of 
Lemma~\ref{le:test} in one direction shows that the right hand side
is non-zero. Using the other part in the other direction then shows that 
$\Gamma_\fp\KInj(kG)$ is minimal. Thus $\KInj(kG)$ is stratified
as a tensor triangulated category by $H^*(G,k)$.

\section{Applications}

In this section, we give some applications of the classification of
localising subcategories of $\StMod(kG)$ and $\KInj(kG)$
(Theorem~\ref{th:strat}). To illustrate
the methods, we give the proof in the case of the tensor product theorem.
The remaining proofs can be found in 
Section~11 of \cite{Benson/Iyengar/Krause:bik3}.

\subsection*{The tensor product theorem}

The tensor product theorem states that if $X$ and $Y$ are 
objects in $\StMod(kG)$, or more generally 
in the larger category $\KInj(kG)$, then
\[ \mcV_G(X\otimes_k Y)=\mcV_G(X)\cap\mcV_G(Y). \]
This was first proved by
Benson, Carlson and Rickard \cite[Theorem~10.8]{Benson/Carlson/Rickard:1996a}
for $\StMod(kG)$ 
in the case where $k$ is algebraically closed. The method of
proof was to reduce to elementary abelian subgroups and then
use the version of Dade's lemma given in Theorem~\ref{th:Dade}.

The proof of the Stratification Theorem~\ref{th:strat} does not
involve any form of Dade's lemma, and so we get a new proof of
the tensor product theorem as follows.  Since 
\[ \Gamma_\fp(X\otimes_k Y)\cong\Gamma_\fp(\one) \otimes_k X \otimes_k Y
\cong \Gamma_\fp(X)\otimes_k Y \cong X \otimes \Gamma_\fp(Y), \]
if either $\Gamma_\fp(X)$ or $\Gamma_\fp(Y)$ is zero then so is
$\Gamma_\fp(X\otimes_k Y)$. This shows that
\[ \mcV_G(X\otimes_k Y)\subseteq \mcV_G(X) \cap \mcV_G(Y). \]
For the reverse containment, suppose that $\fp\in\mcV_G(X)\cap\mcV_G(Y)$.
Thus $\Gamma_\fp(X)\ne 0$ and $\Gamma_\fp(Y)\ne 0$. It follows from
Theorem~\ref{th:strat} and $\Gamma_\fp(X)\ne 0$ that $\Gamma_\fp(\one)$ is in
$\Loc^\otimes(\Gamma_\fp(X))$, and hence that 
$\Gamma_\fp(Y)$ is in 
$\Loc^\otimes(\Gamma_\fp(X \otimes_k Y))$. Since $\Gamma_\fp(Y)\ne 0$, 
this implies that $\Gamma_\fp(X\otimes_k Y)\ne 0$.

\subsection*{The subgroup theorem}

The Subgroup Theorem~\ref{th:subgroup-E} for elementary abelian groups
was proved using the stratification theorem in that context, 
and was used in order
to prove the stratification theorem for general finite groups.
The following general version of the subgroup theorem follows in
the same way from the stratification theorem for finite groups.

\begin{theorem}\label{th:subgroup-G}
Let $H\le G$ be finite groups. If $X$ is an object in $\KInj(kG)$ then
\[ 
\mcV_H(X{\downarrow_H})=(\res^*_{G,H})^{-1}\mcV_G(X). 
\]
\end{theorem}

\subsection*{Thick subcategories}

Theorem~\ref{th:strat} also gives a new proof for the classification
of tensor ideal thick subcategories of $\stmod(kG)$ 
(Theorem~\ref{th:thick}) and $\sfD^b(\mmod(kG))$, 
avoiding the use of the version of Dade's
lemma given in Theorem~\ref{th:Dade}.

\subsection*{Localising subcategories closed under products and duality}

We state the following theorem for $\StMod(kG)$. A similar statement
holds for $\KInj(kG)$.

\begin{theorem}
Under the bijection given in Theorem~\ref{th:localising}, the following
properties of a tensor ideal localising subcategory $\sfC_\mcV$ of
$\StMod(kG)$ are equivalent.
\begin{enumerate}
\item $\sfC_\mcV$ is closed under products.
\item $\mcV_G\smallsetminus\mcV$ is specialisation closed.
\item There exists a set $\mcX$ of finitely generated $kG$-modules
such that $\sfC_\mcV$ is the full subcategory of modules $M$
satisfying $\sHom_{kG}(N,M)$ for all $N$ in $\mcX$.
\item When a $kG$-module $M$ is in $\sfC_\mcV$ so is $\Hom_{k}(M,k)$.
\end{enumerate}
\end{theorem}

\subsection*{The telescope conjecture}

Our final application is a statement about certain localising such
categories which are defined as follows.

\begin{defn}
A localising subcategory $\sfC$ of a triangulated category $\sfT$ is said
to be \emph{strictly localising} if 
there is a localisation functor $L\colon\sfT\to\sfT$
such that an object $X$ is in $\sfC$ if and only if $L(X)=0$.

A localising subcategory $\sfC$ is said to be \emph{smashing} if
it is strictly localising, and 
the localisation functor $L$ preserves coproducts.
\end{defn}

The following is the analogue for $\StMod(kG)$ and $\KInj(kG)$ of
the telescope conjecture of algebraic topology 
(Bousfield \cite{Bousfield:1979a}, Ravenel \cite{Ravenel:1984a,Ravenel:1995a}).

\begin{theorem}
A tensor ideal localising subcategory of $\StMod(kG)$ or of $\KInj(kG)$
is smashing if and only if it is generated by compact objects.
\end{theorem}

\section{Costratification}
\label{sec:Costratification}
Let $\sfT$ be a compactly generated triangulated category with small
products and coproducts. Recall that a \emph{colocalising subcategory}
of $\sfT$ is a thick subcategory that is closed under products.

If $\sfT$ is tensor triangulated then such a subcategory $\sfC$ is closed under tensor products with simple modules if and only if it is \emph{Hom closed}, in the sense that for all $X$ in $\sfT$ and all $Y$ in $\sfC$, the function object (see Section~\ref{se:tensor}) $\fHom(X,Y)$ is in $\sfC$.

The main theorem of \cite{Benson/Iyengar/Krause:bik4} classifies the Hom closed colocalising subcategories of $\StMod(kG)$ and of $\KInj(kG)$. 

\begin{theorem}
\label{th:coloc}
There is a one to one correspondence between Hom closed colocalising subcategories $\sfD$ of $\StMod(kG)$ (respectively $\KInj(kG)$) and tensor ideal localising subcategories $\sfC$ of $\StMod(kG)$ (respectively $\KInj(kG)$) given by $\sfC={^\perp}\sfD$, the full subcategory of objects $X$ satisfying $\Hom^*(X,Y)=0$ for all $Y$ in $\sfD$.
\end{theorem}

The proof of this theorem goes via the notions of cosupport and costratification. We explain these concepts and  fix to this end an $R$-linear tensor triangulated category $\sfT$, as in Section~\ref{se:tensor}. In addition to the axioms listed there, we assume also $\Hom(?, Y )$ is exact for each object $Y\in\sfT$; all tensor triangulated categories encountered in this work have this property.

For each prime $\fp\in\Spec R$ denote by $\Lambda^\fp$ the right adjoint of the functor $\Gamma_\fp$ which exists by the Brown representability theorem.

\begin{defn}
If $X$ is an object in $\sfT$ then the \emph{cosupport} of $X$ is
the subset
\[ \cosupp_R(X)=\{\fp\in\Spec R \mid \Lambda^\fp(X)\ne 0\}. \]
\end{defn}

Note that $\Lambda^\fp$ and $\Gamma_\fp$ provide mutually inverse
equivalences between $\Gamma_\fp(\sfT)$ and $\Lambda^\fp(\sfT)$. Thus
$\cosupp_R(X)$ is a subset of $\supp_R(\sfT)$.  

We say that the tensor triangulated category $\sfT$ is
\emph{costratified} by $R$ if for each $\fp\in\Spec R$ the category
$\Lambda^\fp(\sfT)$ is either zero or minimal among all Hom closed
colocalising subcategories of $\sfT$.

As for localising subcategories, the classification of colocalising
subcategories can be achieved one prime at a time. The way to express
this is via the following maps:
\[ \left\{\begin{gathered}
\text{Colocalising}\\ \text{subcategories of $\sfT$}
\end{gathered}\; \right\} 
\xymatrix@C=3pc {\ar@<1ex>[r]^-{{\sigma}} & \ar@<1ex>[l]^-{{\tau}}}
\left\{
\begin{gathered}
  \text{Families $(\sfS(\fp))_{\fp\in\supp_R(\sfT)}$ with $\sfS(\fp)$ a}\\
  \text{colocalising subcategory of $\Lambda^\fp(\sfT)$}
\end{gathered}\;\right\} \] 
where $\sigma(\sfS)=(\sfS\cap\Lambda^\fp(\sfT))$ and 
$\tau(\sfS(\fp))=\Coloc_\sfT(\sfS(\fp)\mid\fp\in\supp_T(\sfT))$. The following
is \cite[Corollary~9.2]{Benson/Iyengar/Krause:bik4}.

\begin{theorem}
Let $\sfT$ be an $R$-linear tensor triangulated category. If $\sfT$ is
costratified by $R$ then the maps $\sigma$ and $\tau$  
establish a bijection between the Hom closed
colocalising subcategories of $\sfT$ and the subsets of $\supp_R(\sfT)$.
\end{theorem}

From this result one deduces Theorem~\ref{th:coloc} by proving
costratification first for $\KInj(kG)$ (see
\cite[Theorem~11.10]{Benson/Iyengar/Krause:bik4}) and then for
$\StMod(kG)$ (see \cite[Theorem~11.13]{Benson/Iyengar/Krause:bik4}).

Let us include an application which justifies the study of support and cosupport; it is 
 \cite[Corollary~9.6]{Benson/Iyengar/Krause:bik4}.

\begin{theorem}
Suppose the tensor triangulated category $\sfT$ is generated
by its unit. Then $\sfT$ is stratified by $R$ if and only if for all
objects $X$ and $Y$ in $\sfT$ one has
\[
\Hom_\sfT^*(X,Y)=0\quad\iff\quad\supp_R(X)\cap\cosupp_R(Y)=\varnothing.
\]
\end{theorem}
Note that $\StMod(kG)$ (respectively $\KInj(kG)$) is generated by its
unit if and only if $G$ is a $p$-group. We refer to
\cite{Benson/Iyengar/Krause:bik4} for more general results which do
not depend on the fact that $\sfT$ is generated by its unit.

\providecommand{\bysame}{\leavevmode\hbox to3em{\hrulefill}\thinspace}
\providecommand{\MR}{\relax\ifhmode\unskip\space\fi MR }
\providecommand{\MRhref}[2]{%
  \href{http://www.ams.org/mathscinet-getitem?mr=#1}{#2}
}
\providecommand{\href}[2]{#2}

\end{document}